\documentclass[11pt]{amsart}
\usepackage{stmaryrd}
\usepackage[utf8]{inputenc}
\usepackage{amssymb,amsmath,amsthm,enumerate,enumitem,colonequals,mlmodern,tikz-cd,microtype,dirtytalk,hyperref}

\theoremstyle{plain}
\newtheorem{theorem}{Theorem}[section]
\newtheorem*{theorem*}{Theorem}
\newtheorem{lemma}[theorem]{Lemma}
\newtheorem{corollary}[theorem]{Corollary}
\newtheorem{proposition}[theorem]{Proposition}
\theoremstyle{definition}
\newtheorem{defn}[theorem]{Definition}
\newtheorem{remark}[theorem]{Remark}
\newtheorem{example}[theorem]{Example}
\newtheorem{criterion}[theorem]{Criterion}
\newtheorem*{ack}{Acknowledgements}{}

\newcommand{\mysetminusD}{\hbox{\tikz{\draw[line width=0.6pt,line cap=round] (3pt,0) -- (0,6pt);}}}
\newcommand{\mysetminusT}{\mysetminusD}
\newcommand{\mysetminusS}{\hbox{\tikz{\draw[line width=0.45pt,line cap=round] (2pt,0) -- (0,4pt);}}}
\newcommand{\mysetminusSS}{\hbox{\tikz{\draw[line width=0.4pt,line cap=round] (1.5pt,0) -- (0,3pt);}}}

\newcommand{\mysetminus}{\mathbin{\mathchoice{\mysetminusD}{\mysetminusT}{\mysetminusS}{\mysetminusSS}}}

\def\mcO{\mathcal{O}}
\def\mcI{\mathcal{I}}

\def\sfD{\mathsf{D}}
\def\sfE{\mathsf{E}}

\def\sfG{\mathsf{G}}
\def\sfH{\mathsf{H}}

\def\sfJ{\mathsf{J}}
\def\sfK{\mathsf{K}}
\def\sfL{\mathsf{L}}
\def\sfM{\mathsf{M}}

\def\sfS{\mathsf{S}}
\def\sfT{\mathsf{T}}

\def\unit{\mathbf{1}}

\DeclareMathOperator{\Spc}{Spc}
\DeclareMathOperator{\Spec}{Spec}

\DeclareMathOperator{\supp}{supp}

\DeclareMathOperator{\loc}{\mathrm{loc}}

\DeclareMathOperator{\thick}{thick}
\DeclareMathOperator{\Thick}{Thick}
\DeclareMathOperator{\Modu}{\mathsf{Mod}}
\DeclareMathOperator{\Hom}{Hom}
\DeclareMathOperator{\RHom}{RHom}
\DeclareMathOperator{\gr}{gr}

\newcommand{\twoCAlg}{\mathrm{2CAlg}}
\newcommand{\twoCAlgrig}{\twoCAlg^{\mathrm{rig}}}

\newcommand{\Catperf}{\Cat^{\textnormal{perf}}}
\newcommand{\Cat}{\mathrm{Cat}_{\infty}}

\newcommand{\CAlg}{\operatorname{CAlg}}

\newcommand{\Ind}{\operatorname{Ind}}

\DeclareMathOperator{\QCoh}{QCoh}

\title{Distributivity, affineness, and the structure sheaf}

\author{Andy Jiang}
\address{Andy Jiang, Institute of Mathematics, Academia Sinica, 6F, Astronomy-Mathematics Building, No.\ 1, Sec.\ 4, Roosevelt Road, Da-an, Taipei
106319, Taiwan}
\email{jianga2718@gate.sinica.edu.tw}

\author{Greg Stevenson}
\address{Greg Stevenson, Aarhus University, Department of Mathematics, \newline Ny Munkegade 118, bldg.\ 1530
DK-8000 Aarhus C, Denmark
}
\email{greg@math.au.dk}

\begin{document}

\begin{abstract}
\noindent We observe that for a quasi-compact and quasi-separated scheme the structure sheaf generates the perfect complexes if and only if the lattice of thick subcategories is distributive if and only if the affinization map is 0-affine.
\end{abstract}

\maketitle

\section{Introduction}

Hypercohomology is a fundamental invariant of a complex of $\mcO_X$-modules on a scheme $X$. It is thus a natural, and seemingly innocent, question to ask exactly how much of the unbounded derived category $\sfD(X)$ of $\mcO_X$-modules with quasi-coherent cohomology it sees. In other words, how large is the kernel of $\RHom(\mcO_X, -)$? In this paper we consider the extreme case that this kernel is trivial; for such a scheme $X$ hypercohomology witnesses the non-triviality of every complex in $\sfD(X)$ that is not acyclic. This is equivalent to $\sfD(X)$ being generated by the structure sheaf $\mcO_X$ under coproducts and to $\sfD^\mathrm{perf}(X)$ being classically generated by $\mcO_X$.

There is an obvious class of schemes for which this is the case, namely the affine schemes. As we all know, in this case $\RHom(\mcO_X,-)$ gives an equivalence of $\sfD(X)$ with the derived category $\sfD(\mcO_X(X))$ of modules over the global sections. A moments consideration leads one to notice that hypercohomology must also detect vanishing on quasi-affine schemes, i.e.\ quasi-compact open subschemes of affine schemes. It is then natural to ask if generation by the structure sheaf characterises the quasi-affine schemes. This is false, but as we show in Theorem~\ref{affinevsgen} the derived analogue is valid: generation by $\mcO_X$ is equivalent to the derived pushforward along the affinization morphism $a\colon X\to Y = \Spec \mcO_X(X)$ inducing an equivalence $\sfD(X) \cong \Modu_{\sfD(Y)}a_*\mcO_X$ so that $X$ is quasi-affine in a weak sense (a statement which should already be known to experts).

Another invariant of $\sfD(X)$ is $L = \Thick(\sfD^\mathrm{perf}(X)$) the complete lattice of thick subcategories of the perfect complexes. For non-affine schemes this lattice is usually not distributive, for instance this already happens for any smooth projective curve. However, a consequence of tensor triangular geometry is that if $\mcO_X$ generates the perfect complexes then $L$ is a distributive lattice. For general small symmetric monoidal stable $\infty$-categories it is possible for the lattice of thick subcategories to be distributive without the unit generating\textemdash{}for instance this can happen in modular representation theory. We show in Theorem~\ref{thm:distributive} that in algebraic geometry this lattice can only be distributive for the trivial reason: distributivity of $L$ actually forces $\mcO_X$ to be a generator. In fact, our argument applies much more generally and gives a characterization of generating sets in terms of distributivity for small symmetric monoidal stable $\infty$-categories where the unit generates locally with respect to the spectrum.

We then discuss examples, both classical and derived, as well as some consequences. The main consequence is a short proof that for a proper variety of positive dimension the structure sheaf is never a generator. This is proved in Corollary~\ref{cor:proper} as a consequence of the analogous statement, Theorem~\ref{thm:proper}, for proper morphisms.

Let us conclude the introduction by mentioning related work by Ito and Olander \cite{IO25}. Their work gives a very natural criterion for the tensor powers of a line bundle to generate the derived category of a noetherian scheme. In particular, it can be applied to the structure sheaf to give another characterization, in terms of the existence of certain sections over integral closed subschemes, of generation by the structure sheaf.

\begin{ack}
We are grateful to Rudradip Biswas, Benjamin Briggs, and Martin Gallauer for helpful comments on an earlier version and to Kabeer Manali Rahul for his interest and useful discussions. Thanks are due to Daigo Ito and Noah Olander for alerting us to their work on the subject and for the resulting helpful conversations. GS was supported by the Danmarks Frie Forskningsfond (grant ID: 10.46540/4283-00116B) during some of the period in which this work was conducted. GS also thanks Michel Van den Bergh who, many years ago, asked him about when $\mcO_X$ generates the perfect complexes.
\end{ack}


\section{Preliminaries}

Let us make some brief recollections about the terminology and notation we will need. We will use $X$ to denote a scheme, and our schemes are always assumed to be quasi-compact and quasi-separated. We denote the structure sheaf by $\mcO_X$ and use $\sfD^\mathrm{perf}(X)$ to denote the perfect complexes and $\sfD(X)$ to denote the derived category of $\mcO_X$-modules with quasi-coherent cohomology. We view these as stable $\infty$-categories (cf.\ the discussion below).

We will use the notation and terminology of \cite{HZG}. We denote by $\Catperf$ the symmetric monoidal $\infty$-category of small idempotent complete stable $\infty$-categories and exact functors; the symmetric monoidal structure corresponds to (a small variant of) the Lurie tensor product i.e.\ it is universal for multilinear exact functors. The $\infty$-category of \emph{commutative $2$-rings} is defined to be
\[
\twoCAlg = \CAlg^{\otimes}(\Catperf)
\]
the $\infty$-category of commutative algebras in $\Catperf$ with morphisms the symmetric monoidal exact functors. One can think of such a gadget as an $\infty$-categorical enhancement of a tt-category. We will not consider monoidal categories which are not symmetric and so we will often just call these categories $2$-rings.

We will also consider the full subcategory $\twoCAlgrig$ of rigid $2$-rings, i.e.\ those $2$-rings whose homotopy category is a rigid tt-category.

For a $2$-ring $\sfK$ we will use $\unit$ to denote the unit object for the monoidal structure. For an object $k\in \sfK$ we denote by $\thick(k)$ and $\thick^\otimes(k)$ the thick subcategory and thick tensor ideal it generates. We use $\loc(k)$ to denote the smallest localizing subcategory generated by an object $k$ of $\Ind(\sfK)$ and $\loc^\otimes(k)$ for the localizing tensor ideal it generates. As mentioned our primary concern is the rigid $2$-ring $\sfK = \sfD^\mathrm{perf}(X)$ and its corresponding ind-completion $\Ind(\sfK) = \sfD(X)$.

We use $\Thick(\sfK)$ and $\Thick^\otimes(\sfK)$ to denote the partially ordered sets of thick subcategories and thick ideals respectively. Both of these posets are complete lattices, i.e.\ every collection of elements $\mathcal{T}$ admits a supremum and infimum which are given by the same formulas
\[
\bigvee\mathcal{T} = \thick\left( \bigcup_{\sfT \in \mathcal{T}} \sfT\right) \text{ and } \bigwedge\mathcal{T} = \bigcap_{\sfT \in \mathcal{T}} \sfT
\]
in both $\Thick(\sfK)$ and $\Thick^\otimes(\sfK)$. We obviously have the inclusion of posets $\Thick^\otimes(\sfK) \subseteq \Thick(\sfK)$, but it may be strict. If the unit generates $\sfK$, i.e.\ we have $\thick(\unit) = \sfK$ then there is an equality $\Thick^\otimes(\sfK) = \Thick(\sfK)$

\begin{remark}
In the above discussion we should also really consider the radical tensor ideals as well, but our main focus will be the rigid setting in which case one can safely ignore this condition.
\end{remark}

The lattice $\Thick(\sfK)$ is \emph{distributive} if for any three thick subcategories $\sfM_1,\sfM_2,\sfM_3$ we have
\[
\sfM_1 \wedge (\sfM_2 \vee \sfM_3) = (\sfM_1 \wedge \sfM_2) \vee (\sfM_1 \wedge \sfM_3)
\]
or equivalently the dual equation always holds
\[
\sfM_1 \vee (\sfM_2 \wedge \sfM_3) = (\sfM_1 \vee \sfM_2) \wedge (\sfM_1 \vee \sfM_3).
\]

Further information on lattices and their relevance to tt-geometry can be found in \cite{ttnotes}.

%
%
%
%

\section{Tricks for manipulating subcategories}

We begin with some general observations that apply for arbitrary commutative $2$-rings and then apply these in the context of algebraic geometry. In fact, these are all things that take place in the homotopy category and are valid for plain tt-categories.

\begin{lemma}\label{lem:trick1}
Let $\sfK$ be a $2$-ring and $\sfL$ a thick subcategory of $\sfK$. If $\sfM$ is a thick ideal of $\sfK$ such that $\unit$ generates $\sfK/\sfM$ then $\sfL \vee \sfM$ is an ideal.
\end{lemma}
\begin{proof}
We have poset isomorphisms
\begin{align*}
\{\sfJ \in \Thick^\otimes(\sfK) \mid \sfM \subseteq \sfJ\} &\cong \Thick^\otimes(\sfK/\sfM) \\
&= \Thick(\sfK/\sfM) \\
&\cong \{\sfE \in \Thick(\sfK) \mid \sfM \subseteq \sfE\}
\end{align*}
the first because $\sfM$ is an ideal, the second because $\unit$ generates $\sfK/\sfM$, and the last because $\sfM$ is thick. Moreover, the middle identification is an honest equality and the assignments giving the first and last bijections are the same\textemdash{}one just takes the preimage along the quotient functor. Thus the thick ideals containing $\sfM$ are precisely the thick subcategories containing $\sfM$.
\end{proof}

\begin{lemma}\label{lem:trick2}
Let $\sfK$ be a rigid $2$-ring and $\sfM$ a thick ideal of $\sfK$ such that $\unit$ generates $\sfK/\sfM$. Then if $A\in \sfK$ with $\supp A = \Spc \sfK$ we have
\[
\thick(A, \sfM) = \sfK.
\]
\end{lemma}
\begin{proof}
By Lemma~\ref{lem:trick1} we see that $\thick(A, \sfM)$ is an ideal of $\sfK$, which is automatically radical since $\sfK$ is rigid. By hypothesis $A$ has full support and so this ideal must be $\sfK$.
\end{proof}

We will also need the following standard result about presentable stable $\infty$-categories.

\begin{lemma}\label{lem:trick3}
Let $F\colon \sfS \to \sfT$ be an exact functor of presentable stable $\infty$-categories with right adjoint $G$. Then $G$ is conservative if and only if $F$ preserves generators, i.e.\ it sends any generating set for $\sfS$ to a generating set for $\sfT$. Moreover, if $F$ sends some generating set of $\sfS$ to a set of generators for $\sfT$ then this is true for any generating set, that is $F$ preserves generation.
\end{lemma}
\begin{proof}
Assume that $G$ is conservative, which is equivalent to $\ker G = 0$. We see that if $s$ is a generator for $\sfS$ then for any $0\neq t\in \sfT$ the object $Gt$ is also non-zero by assumption and so there is some $i\in \mathbb{Z}$ with
\[
0\neq \sfS(s,\Sigma^i Gt) \cong \sfT(Fs,\Sigma^i t).
\]
Thus $Fs$ witnesses that $t$ is non-zero as required. It is enough to treat the case of a single generator as we can take the coproduct of a set of generators to get a single generator and $F$ preserves coproducts.

On the other hand, suppose that $s$ generates $\sfS$. If $Fs$ generates $\sfT$ then the adjunction formula, as above, shows that $Gt$ can only vanish if $t=0$. In particular, as soon as $F$ sends some generator of $\sfS$ to a generator for $\sfT$ then $G$ is conservative and so $F$ preserves all generating sets by the first part of the argument. 
\end{proof}

\begin{remark}
Actually, if one looks at the argument, in order to deduce that $G$ is conservative we only needed that for some object $s\in \sfS$, not necessarily a generator, its image $Fs$ is a generator for $\sfT$. But we can just add a generator for $\sfS$ to $s$ to get that $F$ preserves some generator, and then once we know $G$ is conservative we know it preserves generation in general.
\end{remark}

We now consider the special case of the perfect and unbounded derived categories of a scheme $X$. Recall that $X$ is, at bare minimum, assumed to be quasi-compact and quasi-separated. All functors we consider are derived, but we do not indicate this explicitly in the notation. (Of course some arguments are formal, depending mostly on the projection formula and tt-geometry concepts, but our interest is schemes so we leave such generalizations to the interested reader.)

For a Thomason subset $V\subseteq X$, i.e.\ a union of closed subsets each having quasi-compact open complement, we denote by $\sfD_V^\mathrm{perf}(X)$ the tensor ideal of perfect complexes supported on $V$ and by $\sfD_V(X)$ its ind-completion. 

\begin{lemma}\label{lem:gencriterion1}
Let $E$ be an object of $\sfD(X)$. Then $E$ generates $\sfD(X)$ if and only if there exists a closed subset $V$ with affine open complement $U$ such that $E\vert_U$ generates $\sfD(U)$ and $\sfD_V(X) \subseteq \loc(E)$. 
\end{lemma}
\begin{proof}
If $E$ is a generator then both assertions are obvious (cf.\ Lemma~\ref{lem:trick3} for the former). Hence suppose we are given a $V$ as in the statement. Because $\sfD_V(X) \subseteq \loc(E)$, the localizing subcategory $\loc(E)$ is determined by its image $\loc(E)/\sfD_V(X)$ in the localization $\sfD(U) = \sfD(X)/\sfD_V(X)$. By assumption this is all of $\sfD(U)$ and so $\loc(E) = \sfD(X)$.
\end{proof}

\begin{remark}
Of course, if the above holds for one such $V$ then it holds for all such $V$.
\end{remark}

\begin{remark}
Because affine schemes are quasi-compact such a $V$ is necessarily a Thomason closed subset.
\end{remark}

\begin{remark}
In particular, if $E$ is vector bundle (non-zero on every connected component) then we see that $E$ generates precisely when it builds all complexes with cohomology supported on the complement of an affine open.
\end{remark}

\begin{lemma}\label{lem:pushgen}
Let $X$ be a noetherian scheme and $i\colon V\to X$ a closed subscheme of $X$. If $E\in \sfD(V)$ is a generator for $\sfD(V)$ then in $\sfD(X)$ we have $\loc(i_*E) = \sfD_V(X)$.
\end{lemma}
\begin{proof}
It follows from the projection formula that $\loc(i_*E)$ is a tensor ideal: if $F\in \sfD(X)$ then we have
\[
F \otimes i_*E \cong i_*(i^*F \otimes E) \in i_*\sfD(V) = i_*\loc(E) \subseteq \loc(i_*E).
\]
Because $E$ is a generator for $\sfD(V)$ we have $\supp E = V$ and it then follows from the classification of localizing ideals that $\loc(i_*E) = \sfD_V(X)$.
\end{proof}

One can get away from the noetherian hypothesis as long as one is careful about the closed subscheme.

\begin{lemma}\label{lem:pushgen2}
Let $X$ be a quasi-compact and quasi-separated scheme and $i\colon V\to X$ a closed subscheme of $X$. Suppose that the sheaf of ideals $\mcI$ corresponding to $V$ is affine locally generated by finitely many elements. Then if $E\in \sfD(V)$ is a generator we have $\loc(i_*E) = \sfD_V(X)$.
\end{lemma}
\begin{proof}
We proceed as in Lemma~\ref{lem:pushgen} and to conclude we just need an argument that the ideal $\loc(i_*E)$ is really $\sfD_V(X)$. For any affine open subscheme $U$ of $X$ the localizing ideal $\sfD_{V\cap U}(U)$ is generated by the Koszul complex on the ideal $\mcI\vert_U$ cutting out $V\cap U$ which we have assumed is finitely generated. This Koszul complex lies in $\loc(i_*E\vert_U)$ because its cohomology is killed by $\mcI\vert_U$. Hence affine locally $i_*E$ generates $\sfD_V(X)$ and so, by the usual descent argument, it also generates it globally. 
\end{proof}

\begin{example}\label{ex:oddball}
    The hypothesis in Lemma~\ref{lem:pushgen2} is necessary. If we take 
    \[
    R = \frac{k[x_2,x_3,x_4,\cdots]}{(x_2^2, x_3^3, x_4^4,\cdots)}
    \]
    as in \cite{NeemanOddball} then $X=\Spec R$ is quasi-compact and separated and the residue field $k$ generates $\sfD(X^\mathrm{red})$ but by Neeman's results in \emph{loc.\ cit.\@} it does not generate $\sfD(X)$. 
\end{example}

\begin{lemma}\label{lem:gencriterion2}
Let $X$ be a noetherian scheme and let $E\in \sfD(X)$ have full support. Then $E$ generates $\sfD(X)$ if and only if there exists a closed subscheme $i\colon V\to X$ with affine open complement $U$ such that:
\begin{itemize}
\item[(i)] $i^*E$ generates $\sfD(V)$;
\item[(ii)] $i_*i^*E \in \loc(E)$.
\end{itemize}
\end{lemma}
\begin{proof}
If $E$ generates $\sfD(X)$ then $i^*E$ generates $\sfD(V)$, for any closed subscheme $i\colon V\to X$, because $i_*$ is conservative (see Lemma~\ref{lem:trick3}). It is also clear that $i_*i^*E \in \loc(E)$ because $\loc(E) = \sfD(X)$.

Suppose then that $i\colon V\to X$ and $E$ verify conditions (i) and (ii) and the complement $U$ of $V$ is open affine. Because $i^*E$ generates $\sfD(V)$ Lemma~\ref{lem:pushgen} tells us that $i_*i^*E$ generates $\sfD_V(X)$. The second condition then says $\sfD_V(X) \subseteq \loc(E)$ and so Lemma~\ref{lem:gencriterion1} finishes the proof; it applies since $\supp E = X$ implies that $E|_U$ generates $\sfD(U)$ (since $U$ is affine).
\end{proof}

Using Lemma~\ref{lem:pushgen2} we can also prove a non-noetherian version with an additional hypothesis on $V$.

\begin{lemma}\label{lem:gencriterion2.5}
Let $X$ be a quasi-compact and quasi-separated scheme and let $E\in \sfD(X)$ have full support. Then $E$ generates $\sfD(X)$ if and only if there exists a closed subscheme $i\colon V\to X$ with open affine complement $U$ such that:
\begin{itemize}
\item[(i)] $i^*E$ generates $\sfD(V)$;
\item[(ii)] $i_*i^*E \in \loc(E)$.
\item[(iii)] $V$ is affine locally cut out by a finitely generated ideal.
\end{itemize}
\end{lemma}

\begin{example}
    Consider $X = \Spec R$ as in Example~\ref{ex:oddball}. Then $X^\mathrm{red} = \Spec k$ has open affine complement (it is empty), $E=k$ has full support and (i) and (ii) are satisfied: any non-zero object generates $\sfD(k)$ and every localizing subcategory of $\sfD(R)$ is an ideal. But, as we already noted, $k$ is not a generator for $\sfD(R)$.
\end{example}


\section{Distributivity}

In this section we make the connection between generation by $\mcO_X$ and distributivity of the lattice of thick subcategories. It is a straightforward combination of standard facts to see that generation by the structure sheaf implies the lattice of thick subcategories is very well behaved.

\begin{proposition}
If $\sfD^\mathrm{perf}(X)$ is generated by $\mcO_X$ then $\Thick(\sfD^\mathrm{perf}(X))$ is a coherent frame, and so in particular is distributive.
\end{proposition}
\begin{proof}
By \cite[Theorem~3.1.9]{KP17} the lattice of radical tensor ideals of $\sfD^\mathrm{perf}(X)$ is a coherent frame and so \emph{a fortiori} distributive. Because $\sfD^\mathrm{perf}(X)$ is rigid every tensor ideal is radical, and because $\sfD^\mathrm{perf}(X) = \thick(\mcO_X)$ every thick subcategory is a tensor ideal. Hence the lattices of radical tensor ideals and of thick subcategories are identified and the claim follows.
\end{proof}

Surprisingly, distributivity is also a sufficient condition for generation. The original argument for the following theorem had an unnecessary noetherian hypothesis; we are grateful to Martin Gallauer for suggesting that it could be removed. 

\begin{theorem}\label{thm:distributive}
Let $X$ be a quasi-compact and quasi-separated scheme. If the lattice $\Thick(\sfD^\mathrm{perf}(X))$ is distributive then $\thick(\mcO_X) = \sfD^\mathrm{perf}(X)$.
\end{theorem}
\begin{proof}
We suppose that $\Thick(\sfD^\mathrm{perf}(X))$ is distributive and let $E$ be a perfect complex. We will show $E\in \thick(\mcO_X)$ by induction on the number of affine open subsets required to cover $\supp(E)$. 

Suppose then that $\supp(E) \subseteq U$ with $U$ open affine and let $Z$ be the closed complement of $U$. By Lemma~\ref{lem:trick2} we have
\[
\thick(\mcO_X) \vee \sfD^\mathrm{perf}_Z(X) = \sfD^\mathrm{perf}(X)
\]
and so
\begin{align*}
\sfD^\mathrm{perf}_{\supp(E)}(X) &= \sfD^\mathrm{perf}_{\supp(E)}(X) \wedge (\thick(\mcO_X) \vee \sfD^\mathrm{perf}_Z(X)) \\
& = (\sfD^\mathrm{perf}_{\supp(E)}(X) \wedge \thick(\mcO_X)) \vee (\sfD^\mathrm{perf}_{\supp(E)}(X) \wedge\sfD^\mathrm{perf}_Z(X)) \\
&= \sfD^\mathrm{perf}_{\supp(E)}(X) \wedge \thick(\mcO_X)
\end{align*}
where we have used that $\sfD^\mathrm{perf}_{\supp(E)}(X) \wedge\sfD^\mathrm{perf}_Z(X)$ vanishes due to the fact that $\supp(E) \cap Z = \varnothing$. Hence $\sfD^\mathrm{perf}_{\supp(E)}(X) \subseteq \thick(\mcO_X)$ from which it is immediate that $E\in \thick(\mcO_X)$. 

Assume then, given a closed Thomason subset $V$, that we know $\sfD^\mathrm{perf}_{V}(X)$ is contained in $\thick(\mcO_X)$ whenever $V$ can be covered by fewer than $n$ affine opens. Let $E$ be a perfect complex such that $\supp(E)$ is contained in a union of $n$ affine open subsets, say $U_1,\ldots, U_{n-1}$ and $U$. Denote by $Z$ the closed complement of the $n$th open affine $U$. Then, as above, we have
\begin{align*}
\sfD^\mathrm{perf}_{\supp(E)}(X) &= \sfD^\mathrm{perf}_{\supp(E)}(X) \wedge (\thick(\mcO_X) \vee \sfD^\mathrm{perf}_Z(X)) \\
& = (\sfD^\mathrm{perf}_{\supp(E)}(X) \wedge \thick(\mcO_X)) \vee (\sfD^\mathrm{perf}_{\supp(E)}(X) \wedge\sfD^\mathrm{perf}_Z(X)) \\
&= (\sfD^\mathrm{perf}_{\supp(E)}(X) \wedge \thick(\mcO_X)) \vee (\sfD^\mathrm{perf}_{\supp(E)\cap Z}(X)).
\end{align*}
By the choice of $Z$ we have
\[
\supp(E) \cap Z \subseteq (U_1\cup \cdots \cup U_{n-1} \cup U) \cap Z = (U_1 \cap Z) \cup \cdots \cup (U_{n-1}\cap Z)
\]
and so the $U_i$ for $1\leq i \leq n-1$ are sufficient to cover $\supp(E)\cap Z$. Thus by the induction hypothesis 
\[
\sfD^\mathrm{perf}_{\supp(E)\cap Z}(X) \subseteq \thick(\mcO_X).
\]
This allows us to simplify the above computation to
\begin{align*}
\sfD^\mathrm{perf}_{\supp(E)}(X) &= (\sfD^\mathrm{perf}_{\supp(E)}(X) \wedge \thick(\mcO_X)) \vee (\sfD^\mathrm{perf}_{\supp(E)\cap Z}(X)) \\
&= \sfD^\mathrm{perf}_{\supp(E)}(X) \wedge \thick(\mcO_X)
\end{align*}
and conclude that $\sfD^\mathrm{perf}_{\supp(E)}(X) \subseteq \thick(\mcO_X)$ as desired. This is sufficient to complete the proof as $X$ is quasi-compact.
\end{proof}

\begin{remark}
If $\sfK$ is any $2$-ring such that $\sfK = \thick(\unit)$ then $\Thick(\sfK)$ is distributive. Indeed, generation by $\unit$ ensures that $\sfK$ is a rigid $2$-ring and it also implies that every thick subcategory is a tensor ideal. Distributivity then follows from the general theory (see \cite{KP17} and \cite{ttnotes} for further information and references).

It is, however, rather miraculous that the converse holds for schemes. This is not true for arbitrary rigid $2$-rings. For instance, the lattice of thick subcategories of the bounded derived category $\sfD^\mathrm{b}(\gr k)$ of $\mathbb{Z}$-graded $k$-vector spaces is distributive but the unit $k$ does not generate this category. 
\end{remark}

Actually the argument we have given applies much more generally; we are grateful to Rudradip Biswas and Benjamin Briggs for questions and suggestions which led to finding the correct generality. Making the obvious alterations one sees that we have actually proved the following theorem.

\begin{theorem}\label{thm:distributive2}
Let $\sfK$ be a rigid $2$-ring and suppose that the spectrum has an open cover $\Spc \sfK = \cup_{i=1}^r U_i$, with corresponding closed complements $Z_i$, such that each $\sfK/\sfK_{Z_i}$ is generated by the unit object. Let $G_\lambda$ for $\lambda \in \Lambda$ be a set of objects of $\sfK$. Then if
\[
L = \{\sfM \in \Thick(\sfK) \mid G_\lambda \otimes \sfM \subseteq \sfM \text{ for all } \lambda \in \Lambda\}
\]
is a distributive lattice we have
\[
\thick(G_{\lambda_1}\otimes \cdots \otimes G_{\lambda_n} \mid \lambda_i \in \Lambda, n\geq 1) = \thick^\otimes(G_\lambda \mid \lambda \in \Lambda).
\]
\end{theorem}
\begin{proof}
We just sketch the modified argument. Let us set 
\[
\sfG = \thick(G_{\lambda_1}\otimes \cdots \otimes G_{\lambda_n} \mid \lambda_i \in \Lambda, n\geq 1) \text{ and } \sfG^\otimes = \thick^\otimes(G_\lambda \mid \lambda \in \Lambda).
\]
To start the induction one can apply Lemma~\ref{lem:trick1} to see that $\sfG \vee \sfK_{Z_i}$ is a tensor ideal and hence contains $\sfG^\otimes$. Thus, since both $\sfG$ and $\sfK_{Z_i}$ are in the distributive lattice $L$ (as is any tensor ideal), if $E\in \sfG^\otimes$ with $\supp E \subseteq U_i$ for some $i$ we have
\begin{align*}
\thick^\otimes(E) &= \thick^\otimes(E) \wedge \sfG^\otimes \\
&\subseteq \thick^\otimes(E) \wedge (\sfG \vee \sfK_{Z_i}) \\
&= \thick^\otimes(E) \cap \sfG
\end{align*}
This shows that $E\in \sfG$. The inductive step is, using the variation above, then essentially the same as in Theorem~\ref{thm:distributive}.
\end{proof}

\begin{remark}
In the above proof one could replace $\thick^\otimes(E)$ by the smallest thick subcategory containing $E$ and closed under tensoring with the $G_\lambda$.
\end{remark}

\begin{remark}
One should again have in mind the work of Ito and Olander \cite{IO25}. In the special case that $\sfK = \sfD^\mathrm{perf}(X)$ and one takes the $G_\lambda$ to be line bundles this gives a distributivity-style criterion for generation by tensor powers of said family of line bundles. However, there is no reason to ask that the $G_\lambda$ be line bundles: one could take the family of objects to consist of a single object $G$ and one then gets a criterion for $\sfD^\mathrm{perf}_{\supp G}(X)$ to be generated by the tensor powers of $G$. It's then natural to ask if there is some necessary and sufficient condition analogous to \cite[Theorem~1.2]{IO25} involving the behaviour of $G$ on integral closed subschemes.
\end{remark}

\begin{remark}
One could formulate an even more general version where one considers an exact endofunctor $F\colon \sfK \to \sfK$ (or even a family of such) and gets a generation criterion in terms of distributivity of the lattice of thick subcategories closed under the action of $F$. 
\end{remark}


\section{0-affine morphisms}

We now give another abstract characterization of generation by the structure sheaf. As above we work with $\infty$-categorical avatars of the derived category, but for what follows it is really crucial.

The key notion is the following definition of Bhatt and Halpern-Leistner \cite[Remark 2.17]{Bhatt/Halpern-Leistner:2017} (we must confess that we have changed the name). 

\begin{defn}
A morphism $f\colon X\to Y$ of schemes is \emph{0-affine} if $f_*$ induces an equivalence
\[
\sfD(X) \cong \Modu_{\sfD(Y)}f_*\mcO_X
\]
between the derived category of $X$ and the category of modules over the sheaf of commutative ring spectra $f_*\mcO_X$.
\end{defn}

\begin{example}\label{ex:0affine}
Suppose that $i\colon U\to Y$ is a quasi-compact open immersion and let $Z$ denote the complement of $U$. Then $i$ is 0-affine. Indeed, in this case $\sfD(U) = \sfD(Y)/\sfD_Z(Y)$ is given by a smashing localization with corresponding idempotent algebra $i_*\mcO_U$ and so, in particular, 
\[
\sfD(U) \cong \Modu_{\sfD(Y)}i_*\mcO_U.
\]

More generally, suppose that $f\colon X\to Y$ is quasi-affine. Then $f$ is 0-affine as everything in sight satisfies descent. 
\end{example}

Actually being $0$-affine is a condition we already know and love.

\begin{proposition}\label{prop:0affineisconservative}
A quasi-compact and quasi-separated morphism $f\colon X\to Y$ of quasi-compact and quasi-separated schemes is $0$-affine if and only if $f_*$ is conservative.
\end{proposition}
\begin{proof}
If $f$ is $0$-affine then $f_*$ is identified with the forgetful functor and so is conservative. On the other hand, suppose that $f_*$ is conservative. Because $f$ is quasi-compact and quasi-separated the functor $f_*$ preserves colimits. Thus the hypotheses of the $\infty$-categorical Barr-Beck theorem \cite[Theorem~3.4.5]{DAGII} are (more than) satisfied and we learn that $\sfD(X)$ is the category of modules over the monad $f_*f^*$ on $\sfD(Y)$. By the projection formula this is the same as the monad induced by the sheaf of rings $f_*\mcO_X$ and so $f$ is 0-affine.
\end{proof}

This already allows us to come to the point as far as generation by the structure sheaf is concerned. We call this a theorem because it is important; it is already known to experts and the proof is quite cheap.

\begin{theorem}\label{affinevsgen}
    Let $a\colon X\to Y = \Spec \mcO_X(X)$ be the affinization morphism. Then $\thick(\mcO_X) = \sfD^\mathrm{perf}(X)$ if and only if $a$ is 0-affine.
\end{theorem}
\begin{proof}
    Suppose that $a$ is 0-affine. Then $a_*$ is conservative (as in Proposition~\ref{prop:0affineisconservative}) and hence, by Lemma~\ref{lem:trick3}, $a^*$ preserves generation. Since $Y$ is affine $\sfD(Y)$ is generated by $\mcO_Y$ and thus $\mcO_X = a^*\mcO_Y$ generates $\sfD(X)$. By virtue of being perfect $\mcO_X$ is then a generator for $\sfD^\mathrm{perf}(X)$.

    In the other direction, suppose that $\mcO_X$ generates $\sfD(X)$. Of course $\mcO_X = a^*\mcO_Y$, and so $a^*$ sends a generator for $\sfD(Y)$ to one for $\sfD(X)$. Thus Lemma~\ref{lem:trick3} tells us that $a_*$ must be conservative (or one can just do the easy computation directly to see that $\mcO_X$ generating means that hypercohomology detects all objects). The affinization is always quasi-compact and quasi-separated and so Proposition~\ref{prop:0affineisconservative} tells us that $a$ is $0$-affine.
\end{proof}

Now let us give a criterion to check a map is $0$-affine locally on the base. This is a further generalization of Example~\ref{ex:0affine} and is essentially stated in \cite[Remark 2.17]{Bhatt/Halpern-Leistner:2017}. Based on what we have seen above it boils down to a locality statement for conservativity of the pushforward.

    %

Let us get warmed up. We start with a map $f\colon X\to Y$. If we are given a quasi-compact open subset $U$ of $Y$ and $Z$ is a subscheme structure on the complement of $U$ which is locally finitely generated then we can form the following derived pullback square
\[
\begin{tikzcd}
V \coprod X_Z \arrow[r, "f'"] \arrow[d, "p'"'] & U\coprod Z \arrow[d, "p"] \\
X \arrow[r, "f"']                              & Y                        
\end{tikzcd}
\]
where $V = f^{-1}(U)$ and $X_Z$ is what it is. 

\begin{lemma}\label{lem:criterion1}
With notation as above, if $f'_*$ is conservative then so is $f_*$.
\end{lemma}
\begin{proof}
We have $p^*f_* \cong f'_*p'^*$ so to show that $f_*$ is conservative it is sufficient to check that $f'_*{p'}^*$ is conservative. The assumption is that $f'_*$ is conservative and so it is enough to check that ${p'}^*$ is conservative, i.e.\ that the pullbacks to $\sfD(V)$ and $\sfD(X_Z)$ are jointly conservative. The kernel of pulling back to $\sfD(V)$ is $\sfD_T(X)$ where $T$ is the closed subset underlying $X_Z$ and so we just need to show that no object supported on $T$ is killed by pulling back along $X_Z \to X$. This can be checked on the classical closed subscheme $i\colon \pi_0(X_Z)\to X$ underlying $X_Z$ and the statement then follows Lemma~\ref{lem:pushgen2}. Indeed, if we had $M\in \sfD_T(X)$ with $i^*M =0$ then for any $E \in \sfD(\pi_0(X_Z))$ we have
\[
M \otimes i_*E \cong i_*(i^*M \otimes E) \cong 0
\]
and so $\ker (M\otimes(-))$ would be a localizing ideal containing $\sfD_T(X)$, by Lemma~\ref{lem:pushgen2} and the above calculation, and $\sfD(V)$ by the support assumption. But these ideals generate $\sfD(X)$ together and we conclude that $M$ must be be trivial.

\end{proof}

To summarize we have proved:

\begin{criterion}\label{ImgonnaCri}
Given a map $f\colon X\to Y$ and a partition $Y = U\coprod Z$ such that $U$ is quasi-compact open and $Z$ is a subscheme defined by a locally finitely generated sheaf of ideals then $f$ is $0$-affine (i.e.\ $f_*$ is conservative) provided both
\[
f^{-1}(U) \to U \text{ and } X \times_Y Z \to Z
\]
are so (where the pullback is derived).
\end{criterion}

One can extend, via induction, to partitions built from suitable open and closed subsets.

\begin{proposition}\label{prop:criterion}
Suppose we have a constructible stratification of $Y$
\[
Y = C_1 \coprod \cdots \coprod C_n
\]
such that each $C_i = U_i \cap Z_i$ where $Z_i$ is a closed subscheme defined by a locally finitely generated sheaf of ideals and $U_i$ is a quasi-compact open subset. Then a map $f\colon X\to Y$ is conservative if the induced maps $f_i\colon C_i \times_Y X \to C_i$ from the derived pullbacks are all conservative (e.g.\ if each $f_i$ is affine). 
\end{proposition}
\begin{proof}
We can first form the pullback square
\[
\begin{tikzcd}
\coprod_i V_i \arrow[r] \arrow[d] & \coprod_i U_i \arrow[d] \\
X \arrow[r]                       & Y                      
\end{tikzcd}
\]
and we see, since the $U_i$ and hence $V_i$ are an open cover, that it's enough to check each $V_i \to U_i$ is conservative. As a set we can write
\[
U_i = C_i \coprod (U_i \mysetminus Z_i)
\]
i.e.\ it is partitioned as in the setup for Criterion~\ref{ImgonnaCri}. Moreover, we have
\[
U_i \mysetminus Z_i = \cup_{j\neq i} C_j \cap U_i.
\]
With this in mind we proceed by induction on the number of terms in the stratification. If we are in the base case that $n=2$ then $U_i \mysetminus Z_i$ is just $C_j \cap U_i$ for $j\neq i$. Thus both of 
\[
C_i \times_{U_i} V_i \to C_i \text{ and } (C_j \cap U_i) \times_{U_i} V_i \to (C_j \cap U_i)
\]
are $0$-affine maps pulled back along an open immersion and so are again $0$-affine. Criterion~\ref{ImgonnaCri} then tells us that $V_i \to U_i$ is $0$-affine and so we conclude that $f$ is $0$-affine.

The inductive step follows similarly: we just need to observe that in general 
\[
U_i \mysetminus Z_i = \bigcup_{j\neq i} C_j \cap U_i
\]
has a constructible partition of size $n-1$ and so we can again apply Criterion~\ref{ImgonnaCri} to the decomposition $U_i = C_i \coprod (U_i \mysetminus Z_i)$.
\end{proof}


\section{Examples and consequences}

There is an obvious class of examples where the structure sheaf generates, namely the quasi-affine schemes. Let us give some more exotic examples. We start with a non-separated curve.

\begin{example}\label{ex:oxgens1}
   Let $X$ denote the affine line with doubled origin. Then $\mcO_X(X) \cong k[x]$ and so the affinization is $a\colon X\to \mathbb{A}^1$. We claim $a_*$ is conservative. There are several ways to check this. One is to write $\mathbb{A}^1 = \mathbb{G}_m \coprod \{0\}$ and use the criterion of the previous section. We have
\[
X \times_{\mathbb{A}^1} \mathbb{G}_m = \mathbb{G}_m
\]
which is affine. So one just needs to compute $X_0 = X\times_{\mathbb{A}^1} \{0\}$ in order to apply Criterion~\ref{ImgonnaCri}. But this is just $\Spec k \coprod \Spec k$ and so again we have an affine map and hence $a$ is $0$-affine and so $\mcO_X$ generates $\sfD^\mathrm{perf}(X)$. In particular, the lattice of thick subcategories is identified with the spatial frame of specialization closed subsets of $X$. 
\end{example}

\begin{remark}
One can also fairly easily compute this directly by using Lemma~\ref{lem:gencriterion2} and considering the cofibre of multiplication by the global section $x\in \mcO_X(X)$.
\end{remark}

\begin{remark}
The scheme $X$ also gives an example of a non-separated scheme such that $\sfD(X)$ is an approximable triangulated category. This is easily deduced from the fact that $\QCoh X$ is a hereditary abelian category.
\end{remark}

It is tempting to hope that such pathological behaviour is restricted to the non-separated setting. After all, quasi-affine schemes are automatically separated. We next give a concrete example of a quasi-projective $k$-scheme which is not quasi-affine and yet is still $0$-affine over $k$.

\begin{example}\label{ex:oxgens2}
Let $z$ be a closed point of $\mathbb{A}^2$ and set
\[
X = \mathrm{Bl}_z\mathbb{A}^2 \mysetminus \{y\}
\]
where $y$ is a point on the exceptional divisor. The global sections of $X$ are $k[u,v]$ so we want to check that $a\colon X \to \mathbb{A}^2$ is $0$-affine. We can again apply the criterion of Criterion~\ref{ImgonnaCri} by taking the partition $\mathbb{A}^2 = \{z\} \coprod (\mathbb{A}^2 \mysetminus \{z\})$. This reduces us to checking that both
\[
X\times_{\mathbb{A}^2}\{z\}  \text{ and } X\times_{\mathbb{A}^2} (\mathbb{A}^2 \mysetminus \{z\}) = \mathbb{A}^2 \mysetminus \{z\}
\]
are $0$-affine. This is clear for the latter as it is quasi-affine. The former is given by the sheaf of formal dg algebras determined by $k[v,\varepsilon]$ with $\vert \varepsilon \vert = -1$ on $\mathbb{A}^1$ and hence is also $0$-affine. 
\end{example}

\begin{remark}
We learned after writing the article that this example also appears as \cite[Lemma~5.30]{IO25} and is originally due to de Jong. Example~\ref{ex:oxgens1} also appears in \emph{loc.\ cit.} as Corollary~5.32.
\end{remark}

\begin{remark}
This is actually a part of a family of examples. If we take an affine scheme $Y$ and a closed subscheme $Z$ and glue the deformation to the normal cone $D \to Y\times \mathbb{A}^1$ to $(Y \times \mathbb{A}^1) \mysetminus (Z \times \{0\})$ along $Y \times (\mathbb{A}^1 \mysetminus \{0\})$ then we can apply the criterion to see the result is $0$-affine.
\end{remark}

Of course all of this makes sense more generally than for just (derived) schemes. 

\begin{example}
The affine line with two origins (Example \ref{ex:oxgens1})
has a natural action by $\mathbb{Z}/2$, and thus descends to a relative algebraic space over $B(\mathbb{Z}/2)$ such that the map to $B(\mathbb{Z}/2)$ is $0$-affine. Working over the base $\mathbb{R}$, we can pullback this example along the nontrivial map $\Spec \mathbb{R} \to B(\mathbb{Z}/2)$ to obtain a non-schematic $0$-affine algebraic space (see \cite{stacks-project} Example 03FN).
\end{example}

\begin{example}
Let $G$ be a unipotent group scheme of finite cohomological dimension over a field $k$, then $BG$ is a $0$-affine geometric stack (hence perfect). The reason is as follows.

A quasi-geometric stack with finite cohomological dimension has a compact structure sheaf \cite[Proposition~9.1.5.3]{lurie2018spectral}. Unipotence implies that the image of the fully faithful functor 
\[D(\Gamma(BG, \mathcal{O}_{BG})) \to \QCoh BG\]
contains all homologically bounded above objects. It is also t-exact because the map $\Gamma(BG, \mathcal{O}_{BG}) \to k$ is coconnectively faithfully flat (see  \cite[Corollary~4.1.12]{luriedagviii}). Therefore the functor is an isomorphism as the domain is left t-complete (because of finite cohomological dimension).
\end{example}

Now let us discuss a case where there are no non-trivial $0$-affine morphisms.

\begin{theorem}\label{thm:proper}
Suppose that $f\colon X\to Y$ is proper and $Y$ is noetherian. Then $f$ is $0$-affine if and only if $f$ is finite.  
\end{theorem}
\begin{proof}
It is clear that if $f$ is finite it is $0$-affine (since it is just affine). For the converse it is enough to check being $0$-affine implies that the fibres of $f$ are finite by \cite[Lemma~02OG]{stacks-project}. So let $\Spec k \to Y$ be a point and let $f'\colon \pi_0(X_k) \to \Spec k$ be the classical fibre. The map $f'$ is also $0$-affine. 
		
Let $Z$ be a connected component of $\pi_0(X_k)$. Then, because $f'$ is $0$-affine, $\sfD^\mathrm{perf}(Z)$ is equivalent to a $\sfD^\mathrm{perf}(A)$ where $A = f'_*(\mcO_{\pi_0(Z)})$ is a proper coconnective dg algebra over $k$ with $\sfH^0(A)= k$. Because $k$ is a field and $A$ is coconnective we may model $A$ by a dg algebra such that $A^{<0} = 0$ and $A^0=k$. 

    Without loss of generality we can replace $Z$ by its reduction. Thus by virtue of being reduced and of finite type over a field there is a regular point $z\in Z$. Hence the dg algebra $A$ has a perfect module $\Hom(A, k(z))$ which has cohomology concentrated in degree $0$. It follows from \cite[Corollary~6.2]{MWdg} that $A$ is quasi-isomorphic to $k$ and hence $Z \cong \Spec k$ is a point. Thus the fibre $\pi_0(X_k)$ is $0$-dimensional as claimed.
\end{proof}

\begin{corollary}\label{cor:proper}
    Let $X$ be a proper scheme over a noetherian ring $R$. Then $\mcO_X$ generates $\sfD^\mathrm{perf}(X)$ if and only if $X$ is $\Spec S$ for a finite $R$-algebra $S$.
\end{corollary}
\begin{proof}
    We only consider the non-trivial direction: suppose that $\mcO_X$ generates the perfect complexes over $X$. Then $a\colon X \to \Spec \mcO_X(X)$ is proper and by Theorem~\ref{affinevsgen} it is $0$-affine. So Theorem~\ref{thm:proper} tells us the morphism $a$ is actually finite (and in fact then $X$ is finite over $\Spec R$) and hence $X$ is affine and given by a finite $R$-algebra.
\end{proof}



\end{document}